\documentclass[12pt,reqno]{amsart}
\usepackage{graphicx}
\usepackage{hyperref}
\usepackage{psfrag}
\usepackage{pxfonts}
\usepackage{mathrsfs}
\usepackage{color}
\usepackage{amsmath,amssymb}
\usepackage{tikz}

\numberwithin{equation}{section}

\theoremstyle{plain}
\newtheorem{maintheorem}{Theorem}

\newcommand{\N}{\mathbb{N}}
\newcommand{\Z}{\mathbb{Z}}

\newcommand{\T}{\mathbb{T}}

\newtheorem{theorem}{Theorem}[section]

\newtheorem{conjecture}[theorem]{Conjecture}
\newtheorem{corollary}[theorem]{Corollary}
\newtheorem{proposition}[theorem]{Proposition}
\newtheorem{lemma}[theorem]{Lemma}
\newtheorem{definition}[theorem]{Definition}

\theoremstyle{remark}

\begin{document}

\thanks{The first author was supported by Capes, process PRAPG n. 88887.196178/2025-00  and  would like to express his sincere gratitude to his former M.Sc. advisor, Professor Nivaldo Costa Muniz, for his guidance and encouragement during his master's studies. The idea of this work originated from the studies  obtained under his supervision.}

\author{J. Santana C. Costa}
\address{DEMAT-UFMA S\~{a}o Lu\'{i}s-MA, Brazil.}
\email{jsc.costa@ufma.br}

\author{F. Micena}
\address{
  IMC-UNIFEI Itajub\'{a}-MG, Brazil.}
\email{fpmicena82@unifei.edu.br}

%\author{A. Tahzibi}
%\address{Departamento de Matem\'atica,
%  ICMC-USP S\~{a}o Carlos-SP, Brazil.}
%\email{tahzibi@icmc.usp.br}

\renewcommand{\subjclassname}{\textup{2000} Mathematics Subject Classification}

\date{\today}

\setcounter{tocdepth}{2}
\title{On Non-Wandering Sets  with Non-empty Interior for Endomorphisms}
%\title{Endomorphisms with Nonempty Interior Non-Wandering Set}
\maketitle
\begin{abstract}
In this paper, we study transitivity for endomorphisms. Our results are related to a conjecture of F. Abdenur, C. Bonatti and L. Díaz \cite{ABD}, concerning the relationship between transitivity and the existence of a non-wandering set with nonempty interior. We obtain transitivity under the assumption that there exists a hyperbolic non-wandering set with non-empty interior, and in the setting of accessible partially hyperbolic endomorphisms.
\end{abstract}

%----------------------------------------------------------
\section{Introduction}
%----------------------------------------------------------

Transitivity is a fundamental topological property in Ergodic Theory and Dynamical Systems and is closely related to chaotic behavior. Informally, a transitive system cannot be decomposed into simpler invariant subsystems. More precisely, a continuous map $f:M\to M$ is \textit{transitive} if for any pair of nonempty open sets $U,V\subset M$, there exists $n\ge1$ such that
$
	f^n(U)\cap V\neq\varnothing.
$
Equivalently, $f$ admits a dense orbit. Despite its central role, several basic questions concerning transitivity remain open. A classical and well-known problem asks whether every Anosov system is transitive.

For diffeomorphisms, significant progress has been made. From the works of J. Franks and A. Manning \cite{F70, M74}, every Anosov diffeomorphism on the torus $\T^d$ is transitive. In the same direction, S. Newhouse \cite{N70} showed that codimension-one Anosov diffeomorphisms are transitive. In particular, every Anosov diffeomorphism of dimension two or three is transitive.

The study of endomorphisms exhibiting some form of hyperbolicity originates from the works of \cite{MP75} and \cite{PRZ} and has recently attracted considerable attention; see, for instance, \cite{CV23}, \cite{HH21}, and \cite{LPPR24}.  In this context, by (Theorem 8.3.5, \cite{AH94}), for any Anosov endomorphism on the $d$-torus, the non-wandering set is the whole manifold, that is,  $\Omega(f)=\T^d$. Combined with the Spectral Decomposition Theorem, this implies that $f$ is transitive. Consequently, all Anosov endomorphisms of the torus $\T^d$ are transitive.

 More recently, X. Zhang \cite{Z23} proved that special Anosov endomorphisms of codimension one are also transitive.  In the same work, it is shown that for special Anosov endomorphisms, the following properties are equivalent:
\begin{enumerate}
	\item $f$ is transitive;
	\item $\Omega(f)=M$;
	\item $f$ is topologically mixing;
	\item every unstable leaf $W^u(\bar{x})$ is dense in $M$.
\end{enumerate}
We observe that, for Anosov endomorphisms, transitivity is not equivalent to the property that every stable leaf $W^s(\bar{x})$ is dense in $M$, as is the case for Anosov diffeomorphisms. Indeed, the author also provides examples of transitive Anosov endomorphisms whose stable leaves are not dense.

Given that the condition $\Omega(f)=M$ is universally equivalent to transitivity in the Anosov endomorphisms context, understanding the general properties of the non-wandering set becomes crucial for our investigation.

Recall that, for a continuous map $f:M\to M$, a point $x\in M$ is called \textit{non-wandering} if for every open neighborhood $V$ of $x$ there exists $n\ge1$ such that
$
f^n(V)\cap V\neq\varnothing.
$
The set of all nonwandering points, denoted by $\Omega(f)$, is always $f$-subinvariant, that is, $f(\Omega(f))\subset\Omega(f)$. If $f$ is a local homeomorphism, then $\Omega(f)$ is invariant: $f(\Omega(f))=\Omega(f)$.
Indeed, suppose $\deg(f)=d$ and let $x\in\Omega(f)$ with
$
f^{-1}(x)=\{y_1,\ldots,y_d\}.
$
Choose $\varepsilon>0$ sufficiently small so that each restriction
$
f_i:B_\varepsilon(y_i)\to f(B_\varepsilon(y_i)),
\,\, i=1,\ldots,d,
$
is a homeomorphism. For each $n\in\N$, define
$
U_n=\bigcap_{i=1}^d f\big(B_{\varepsilon/n}(y_i)\big).
$
Since $x$ is nonwandering, for each $n$ there exists $m\ge1$ such that
$
f^m(U_n)\cap U_n\neq\varnothing.
$
Hence, for some $i\in\{1,\ldots,d\}$,
$
f^{m-1}(U_n)\cap B_{\varepsilon/n}(y_i)\neq\varnothing.
$
By finiteness, there exists $y\in\{y_1,\ldots,y_d\}$ and infinitely many such $m$, implying $y\in\Omega(f)$ and $f(y)=x$.

It is important to note that while the invariance $f(\Omega(f)) = \Omega(f)$ is a general property for local homeomorphisms, the complete invariance of $\Omega(f)$, which requires $f^{-1}(\Omega(f)) = \Omega(f)$, is not always guaranteed in the non-invertible setting. Since it is well established that transitivity implies the complete invariance of the non-wandering set, our investigation focuses on the converse implication: identifying the conditions under which the structure of $\Omega(f)$, such as having a nonempty interior, coupled with its complete invariance, ensures the transitivity of the system.

In this work, we study dynamical systems whose non-wandering set has nonempty interior and investigate conditions under which this implies transitivity.
In this direction, F.~Abdenur, C.~Bonatti, and L.~Díaz, proposed the following conjecture:

\begin{conjecture}[\cite{ABD}]\label{conj}
There exists a residual subset $R$ of $Diff^1(M)$ such that, for any $f \in R$, if $\Omega(f)$ has nonempty interior, then $f$ is transitive.
\end{conjecture}

In the case of endomorphisms, there are difficulties arising from the lack of invertibility are mainly due to the multiplicity of unstable directions at a given point, see, e.g., \cite{MT16}. Our results generally require the assumption of complete invariance of a set. This hypothesis is well known in the literature on non-invertible dynamics (see for instance \cite{LPPR24}, \cite{M-2010} and \cite{UW04}).

\begin{definition}
 A set $\Lambda\subset M$ is called \textit{completely $f$-invariant} if $f(\Lambda)=\Lambda=f^{-1}(\Lambda)$.	
\end{definition}

As claimed in \cite{GH}, Conjecture \ref{conj} is false in the non-invertible setting. In dimension two, one can construct examples of endomorphisms admitting a robustly $C^1$-persistent global hyperbolic attractor that is not the whole manifold; see \cite{BKRU}. It is important to note that the attractors in \cite{BKRU} capture points from outside the set itself, which implies that the attractor is not completely invariant.
Since Theorem \ref{TeoB} does not require the assumption of complete invariance, it provides an interesting counterpart to the examples constructed in \cite{BKRU}. Indeed, Theorem \ref{TeoB} implies that the phenomenon exhibited in \cite{BKRU} cannot occur when hyperbolic repellers are considered in place of hyperbolic attractors.

A question that may be of interest is whether an analogue of Conjecture \ref{conj} holds for endomorphisms whose non-wandering set is completely invariant.

There are several works concerning transitivity and non-wandering sets with nonempty interior in dimension two, such as \cite{BKRU}, \cite{GH}, and \cite{R}, among others. In contrast, our results hold in arbitrary dimension, either in the setting of accessible non-invertible partially hyperbolic endomorphisms or under the assumption that the non-wandering set contains a completely invariant hyperbolic subset with nonempty interior.

\section{Main Results}

Throughout the paper, $M$ denotes a compact, connected, boundaryless $C^{\infty}$ Riemannian manifold. Within this framework, we establish the following results.

\begin{maintheorem}\label{TeoA}
	If $f \colon M \to M$ is a $C^1$ endomorphism and $\Lambda \subset \Omega(f)$ is a completely invariant hyperbolic set with nonempty interior, then $f$ is a transitive Anosov endomorphism.
\end{maintheorem}

As a consequence of Theorem \ref{TeoA}, we obtain a dichotomy for the boundary of the non-wandering set.

\begin{corollary}\label{cor1}
	Let $f:M \to M$ be a $C^1$ Anosov endomorphism such that $\Omega(f)$ is completely invariant, then
	\begin{itemize}
		\item[a)] $\partial \Omega(f)=\Omega(f)$ or
		\item[b)] $\partial \Omega(f)=\varnothing$ and in this case, $f$ is transitive.
	\end{itemize}
\end{corollary}

In particular, in the invertible case, if an Anosov diffeomorphism is not transitive, then $\Omega(f)=\partial \Omega(f)$.
We also obtain the following immediate corollary for Axiom A endomorphisms under hypotheses on completely	 invariant sets.

\begin{corollary}
	Let $f:M \to M$ be an $C^1$ Axiom A endomorphism. If
	\begin{itemize}
		\item[(a)] $f$ has a completely invariant basic set with nonempty interior, or
		\item[(b)] $\Omega(f)$ is completely invariant with nonempty interior,
	\end{itemize}
	then $f$ is transitive.
\end{corollary}

For an Axiom A endomorphism, we say that a basic set $B$ is a \textit{repeller} if $W^s(B)=B$.

The next result gives an alternative condition that does not require the assumption of complete invariance. In particular, the existence of a repeller with nonempty interior ensures global hyperbolicity and transitivity.

\begin{maintheorem}\label{TeoB}
	Let $f:M \to M$ be a $C^1$ Axiom A endomorphism with a repeller with nonempty interior. Then $f$ is a transitive Anosov endomorphism.
\end{maintheorem}

We observe that every Axiom A endomorphism has a repeller. Indeed, since $f$ is an Axiom A endomorphism, by Theorem 3.2.5 of \cite{AH94}, the induced map $\bar{f}$ on the inverse limit space $M^f$ (see definitions in Section \ref{sec-preliminares}) has a repeller $\overline{B}$. Therefore, the set $B=\pi(\overline{B})$ is a repeller for $f$.

\begin{maintheorem}\label{TeoC}
	Let $f:M \to M$ be a $C^1$ accessible, $u$-special partial hyperbolic endomorphism. If  $\Omega(f)$ is completely invariant with  nonempty interior,
	then $f$ is transitive.
\end{maintheorem}

From \cite{J17}, there is an open and dense subset of partially hyperbolic endomorphism with $dim (E^c)=1$ which are accessible. To prove Theorem \ref{TeoC}, it is necessary to use a non-invertible version of Brin's Theorem, see \cite{Br}.

\begin{maintheorem}\label{accBrin}
	Let $f:M \to M$ be an accessible, partially hyperbolic endomorphism. If  $\Omega(f)=M$, then $f$ is transitive.	
\end{maintheorem}

The next theorem is similar to the Theorem \ref{TeoC}, but we do not require it to be $u$-special, instead we ask for a slightly stronger condition in the completely invariant set.

\begin{maintheorem}\label{TeoD}
	Let $f:M \to M$ be a $C^1$ accessible, partial hyperbolic endomorphism and  $\Lambda\subset M$ a completely invariant compact set  such that $f\vert_{\Lambda}:\Lambda\to\Lambda$ is transitive. If $\operatorname{Int}(\Lambda)\neq\varnothing$, then $f$ is transitive.
\end{maintheorem}

%####################################################################

\section{Anosov and partially hyperbolic endomorphisms}\label{sec-preliminares}

We begin by recalling the notion of Anosov endomorphisms, which extend the classical concept of Anosov diffeomorphisms to the non-invertible setting, introduced in the works \cite{MP75} and \cite{PRZ}, whose characteristic is exhibit uniform hyperbolic behavior along every full orbit.

\begin{definition}[\cite{PRZ}]\label{def1} Let $f:M\rightarrow M$ be a $C^1$ local diffeomorphism. The map $f$ is an \emph{Anosov endomorphism} if there are
	constants  $C>1$ and $\lambda>1$, such
	that, for each orbit $(x_n)_{n\in{\Z}}$ of $f,$ that is, $f(x_n)=x_{n+1},$ there is a splitting
	$T_{x_i}M=E_{x_i}^s\oplus E_{x_i}^u,\,\, \forall i\in\Z$ which is $Df$-invariant and for all $n>0$ we have:
	\begin{enumerate}		
		\item $||Df^n_{x_i}(v^u)||\leq C^{-1}\lambda^n||v^u||,$ for $v^u\in E^u_{x_i}$,
		\item $||Df^n_{x_i}(v^s)||\leq C\lambda^{-n}||v^s||,$ for $v^s\in E^s_{x_i}$.
	\end{enumerate}
\end{definition}

As in the case of Anosov diffeomorphisms, this hyperbolic splitting induces local invariant manifolds.

\begin{theorem} [Theorem 2.1, \cite{PRZ}]\label{Teo-Loc-manifolds}
	Let $f$ be an Anosov endomorphism, for each orbit $\bar{x}=(x_n)_{n\in{\Z}}$  the following conditions are satisfied:
	\begin{itemize}
		\item[a)] the set $W^s_{\delta}(\bar{x})=\{y\in M; d(f^n(x),f^n(y))<{\delta}; \forall n\leq 0\}$ is a manifold (called the local stable manifold);
		\item[b)] the set $W^u_{\delta}(\bar{x})=\{y\in M;\exists \, (y_n)_{n=-\infty}^0, f(y_n)=y_{n+1}, \,\mbox{and}\, d(x_n,y_n)<{\delta}, \forall n<0 \}$    is a manifold (called the local unstable manifold);
		\item[c)] $T_{x_0}W^s_R(\bar{x})=E^s(x_0)$ and $T_{x_0}W^u_R(\bar{x})=E^u(x_0)$.
	\end{itemize}
\end{theorem}

Anosov endomorphisms exhibit several specific features that distinguish them from Anosov diffeomorphisms. As in the definition above, given $f: M \rightarrow M$ an Anosov endomorphism and $ \bar{x} = (x_n)_{n \in \mathbb{Z}}, \bar{y} = (y_n)_{n \in \mathbb{Z}} $ two different orbits for $f$ such that $x_0 = y_0,$ then $E^s_f(x_0) = E^s_f(y_0),$ however, it is possible that $E^u_f(x_0) \neq E^u_f(y_0),$ where $E^u_f(x_0)$ is the unstable direction defined by $\bar{x} $ and $E^u_f(y_0)$ is the unstable direction defined by $\bar{y}.$ By  Theorem 2.1 of \cite{PRZ}, these directions are integrable to unstable local discs $W^u(\bar{x})$ and $W^u(\bar{y})$, so a point $x$ can have more than one local unstable manifolds.

Given $f:M\to M,$ we call it  \textit{inverse limit space} the set

$$
M^f=\left\{ (x_n)_{n\in {\Z}} \in  \prod_{i\in{\Z}} M_i; \,\,\, M_i=M \,\, {\rm and} \,\, f(x_n)=x_{n+1} \right\}.
$$

Assuming that $M$ is a compact metric  and $f$ is continuous, it is easy to check that $M^f$ is also a compact metric space when endowed with the metric
$$
\bar{d}(\bar{x},\bar{y})=\sum_{i\in{\Z}}\frac{d(x_i,y_i)}{2^{|i|}},
$$
where $d(\cdot,\cdot)$ denotes the metric on $M$.

We define the shift map $\bar{f}:M^f\rightarrow M^f$ by  $\bar{f}((x_n)_{n\in {\Z}})=(x_{k})_{k\in {\Z}},$ where $k=n+1$. Consider the projection $\pi:M^f\rightarrow M,$  onto the zeroth coordinate, that is, if $\bar{x}=(x_n)_{n\in {\Z}}$ then
$\pi(\bar{x})=x_0$.  One can verify that $\pi$ is continuous.

If $f$ is transitive, $\bar{f}$ is transitive too, see Theorem 3.5.3 of \cite{AH94}.

\begin{theorem}\label{Teo-transitive}
	Let $f:M\to M$ be a continuous surjection, then $f$ is topologically transitive if and only if the shift map $\bar{f}:M^f\rightarrow M^f$ so is.
\end{theorem}

One of the features of Anosov diffeomorphisms is the regular dependence of invariant manifolds on the base point. In the endomorphism setting, this property remains valid when considering the inverse limit space.

\begin{theorem}[Theorem 2.5, \cite{PRZ}]\label{PRZ-1}
	Let $f:M\rightarrow M$ be an $C^1$ Anosov endomorphism, consider $x_n\in M^f$ converging to $\bar{x}\in M^f$, then
	$$
	W^u_{r}(\bar{x}_n)\to W^u_{r}(\bar{x}) \,\,\, \mbox{and} \,\,\, W^s_{r}(\bar{x}_n)\to W^s_{r}(\bar{x})$$
	in the $C^1$ topology.
\end{theorem}

Let a set $\Lambda\subset M,$ we denote
$\Lambda^f=\{(x_n)_{n\in\Z}\in M^f; x_n\in\Lambda,\,\, \mbox{for\,\,all}\,\, n\in\Z\}$.

\begin{definition}
	Let $f:M\rightarrow M$ be a $C^1$ local diffeomorphism. A set $\Lambda\subset M$ is a \emph{hyperbolic set} of $f$ if $\Lambda$ is non-empty compact set  with $f(\Lambda)=\Lambda$ and  for all  \mbox{$\bar{x} \in \Lambda^f$} there is a continuous $Df$-invariant splitting
	$T_{x_i}M=E_{x_i}^s\oplus E_{x_i}^u,$ $\, \forall i\in\Z$ and	constants  $C>1$ and $\lambda>1$, such
	that
	\begin{enumerate}		
		\item $||Df^n_{x_i}(v^u)||\leq C^{-1}\lambda^n||v^u||,$ for $v^u\in E^u_{x_i}$,
		\item $||Df^n_{x_i}(v^s)||\leq C\lambda^{-n}||v^s||,$ for $v^s\in E^s_{x_i}$.
	\end{enumerate}	
\end{definition}

Next, we recall the notion of Axiom $A$ endomorphisms.

\begin{definition}[Definition 3.3, \cite{PRZ}]
	Let $f:M\rightarrow M$ be a $C^1$ local diffeomorphism. The map $f$ is an \emph{Axiom A} endomorphism if:
	\begin{itemize}
		\item $\Omega(f)$ is a hyperbolic set,
		\item $\overline{Per(f)}=\Omega(f)$.
	\end{itemize}
\end{definition}

Anosov endomorphisms automatically satisfy Axiom $A$, as stated in the following proposition.

\begin{proposition}[Proposition 3.2, \cite{PRZ}]\label{PropAA}
	Let $f:M\rightarrow M$ be  an Anosov endomorphism, then
	$\overline{Per(f)}=\Omega(f)$. In particular, every  Anosov endomorphism is an Axiom A  endomorphism.
\end{proposition}

Similarly to the invertible case,  Axioma A  endomorphism satisfy the property of spectral decomposition of the nonwandering set, see  Theorem 3.4.4 of \cite{AH94}. More precisely:

\begin{theorem}[Spectral decomposition of $\Omega(f)$]\label{Teo-Dec-Spec}
	If $f$ is an Axiom $A$ endomorphism, then $\Omega(f)$ can be written in a unique way as a disjoint union
	$\Omega=\bigcup_{i=1}^{l}\Omega_i$, where each $\Omega_i$ is compact, satisfies $f(\Omega_i)=\Omega_i$ and $f$ is transitive on $\Omega_i$. The sets $\Omega_i$ are called the basic sets of $f$. Moreover, each $\Omega_i$ can be further decomposed into a finite disjoint union $\Omega_i=\bigcup_{1\leq j\leq n_i}\Omega_{i,j}$, where $\Omega_{i,j}$ is compact, $f(\Omega_{i,j})=\Omega_{i,j+1}$\,\, ($\Omega_{i,n_i+1}=\Omega_{i,1}$) and $f^{n_i}$ is mixing on each $\Omega_{i,j}$. 	
\end{theorem}

We now turn to partially hyperbolic endomorphisms, which generalize uniform hyperbolicity by allowing an intermediate central direction with weaker dynamical behavior.

\begin{definition}[\cite{CM22}]\label{def 1}
	Let $f:M\rightarrow M$ be a $C^1$ local diffeomorphism. The map $f$ is a \emph{partially hyperbolic endomorphism} if there are constants \mbox{$0<\nu<\gamma_1\leq\gamma_2<\mu$} with $\nu<1,$ $\mu>1$ and $C>1$ such
	that for each $(x_n)_{n\in{\Z}} \in M^f,$ there is a splitting
	$T_{x_n}M=E_{x_n}^s\oplus E_{x_n}^c\oplus E_{x_n}^u$ satisfying:
		\begin{enumerate}		
		\item $Df(E^{\ast}_{x_i})=E^{\ast}_{x_{i+1}}, \ast \in \{s,c,u\},$  for any $i \in \mathbb{Z},$		
		\item $||Df^n_{x_i}(v^s)||\leq C\nu^n||v^s||,$
		\item  $C^{-1}\gamma^n_1||v^c||\leq||Df^n_{x_i}(v^c)||\leq C\gamma^n_2||v^c||,$
		\item $C^{-1}\mu^n||v^u||\leq||Df^n_{x_i}(v^u)||,$	
	\end{enumerate}
	for any $ v^s\in E_{x_i}^s,  v^c\in E_{x_i}^c$ and  $v^u\in E_{x_i}^u.$	
\end{definition}

As in the definition of Anosov endomorphism, partially hyperbolic endomorphisms can present a complicated structure of the center and unstable bundles. Given a point $x\in M$ can be defined infinitely many $E^c, E^u$ bundles at $x,$ see for instance Theorem A of \cite{CM22}. Moreover, $E^{cs}$ is uniquely defined for each point. For more topological properties of partially hyperbolic and Anosov endomorphisms we refer to \cite{CM22} and \cite{MT16}.

\begin{definition}
	Let $f : M \to M$ be a partially hyperbolic endomorphism. The map $f$ is \emph{$u$-special} if for all $x \in M$, the unstable direction $E^u_f(x) \subset T_xM$ is independent of the orbit $\bar{x} = (x_n)_{n\in\mathbb{Z}}$ such that $x_0 = x$. Similarly, $f$ is said to be \emph{$c$-special} or \emph{$cu$-special} if the corresponding property holds for the central and center-unstable directions, respectively.
\end{definition}

An important geometric property in the study of partially hyperbolic systems is accessibility. In general terms, this means that stable and unstable manifolds connect any two points through a finite path of segments, for an illustrative idea see Figure \ref{fig:accessibility}.

\begin{definition}\label{accdef}
	A partially hyperbolic endomorphism $f:M\rightarrow M$ is accessible if for
	any pair of points $x, y \in M$, there exists a piecewise $C^1$ path $\gamma_1\ast\gamma_1\ast\ldots\ast\gamma_n$ connecting
	$x$ and $y$, such that:
	\begin{itemize}
		\item[(1)] Each path $\gamma_j:[0,1]\to M,$ $(j=1,\ldots,n)$ satisfies:
		\begin{itemize}
			\item[$\bullet$] $\gamma_j([0,1])\subset W^s_f(z_{j-1})$ or
			\item[$\bullet$] $\gamma_j([0,1])\subset W^u_f(\bar{z})$, where $\bar{z}$ is a lift of  $z_{j-1}$ (i.e., $\pi(\tilde{z}) = z_{j-1}$).			
		\end{itemize}
		\item[(2)]  The path connects $x = z_0 , \ldots, z_n = y$ with $\gamma_j(0) = z_{j-1}$ and $\gamma_j(1) = z_j$.
	\end{itemize}
\end{definition}
\begin{figure}[h!] % [h!] tenta forçar a posição "here"
%	\centering
	\includegraphics[width=0.8\textwidth]{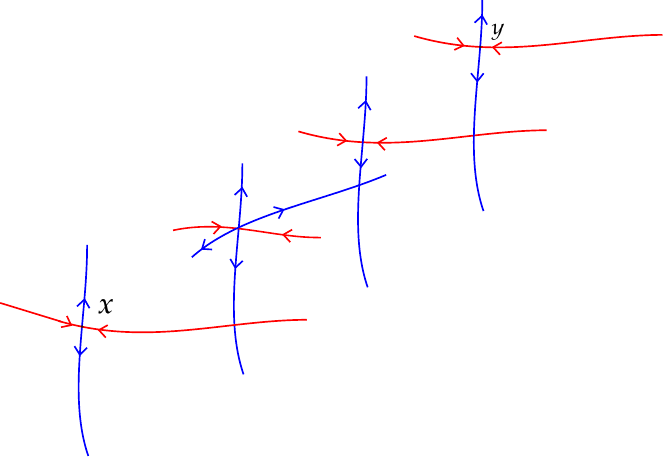} % Ajusta a largura
	\caption{Example of accessibility from $x$ to $y$.}
	\label{fig:accessibility}
\end{figure}

From \cite{J17} we know that there is an open and dense subset of partially hyperbolic endomorphism with $dim (E^c)=1$ that are accessible.

\vspace{0.3cm}

%####################################################################

\section{Proof of theorems \ref{TeoA} and \ref{TeoB}}

\begin{lemma}\label{Lem00}
	Let $f:M \to M$ be a local homeomorphism and $\Lambda$ a compact set completely invariant and denote $U=\operatorname{Int}(\Lambda)$,  then $U$ and its closure $\overline{U}$ are both completely invariant.	
\end{lemma}

\begin{proof}
Suppose $U\neq\varnothing$. Since $\Lambda$ is invariant and  $f$ is a local homeomorphism, then $f(U)\subset U$ and consequently $U \subset f^{-1}(U).$ On the other hand, since $f$ is continuous, $f^{-1}(U)$ is an open set contained in $\Lambda,$ since $\Lambda$ is completely invariant. Therefore also $f^{-1}(U) \subset U,$ as required.

The equality $U = f^{-1}(U),$ and since $f$ is surjective local homeomorphism, we obtain $f(U) = f(f^{-1}(U)) = U.$ Since $\Lambda$ is compact, we have $f(\overline{U})=\overline{f(U)}=\overline{U}$.

Again by compactness of $\Lambda$, $f^{-1}(\overline{U})$  is compact which contains $U$, and thus contains $\overline{U}.$ Hence $\overline{U}\subset f^{-1}(\overline{U})$. Finally, in order to show $f^{-1}(\overline{U})\subset \overline{U},$ take $x\in f^{-1}(\overline{U}),$ so there is $y\in \overline{U}$ with $x\in f^{-1}(\{y\})$. There are neighborhoods  $V_x$ of $x$ and $V_y$ of $y$ such that $f|_{V_x}:V_x\to V_y$ is a homeomorphism. Since $y\in \overline{U},$ there is $y_n\in V_y\cap U$ such that $y_n\to y.$  The continuity of $(f|_{V_x})^{-1}$ we have $(f|_{V_x})^{-1}(y_n)\to (f|_{V_x})^{-1}(y) = x,$ therefore $x\in \overline{U}.$
\end{proof}

\begin{proof}[Proof of Theorem A]

Denote by $U = \operatorname{int}(\Lambda)$.	
Up to replacing $U$ with one of its connected components, we can suppose that $\dim E^s_x$ is constant, $\forall x \in \overline{U},$
it makes constant the function $(\bar x,i) \mapsto \dim E^u_{\bar x_i} > 0,$ $\forall i \in \mathbb{Z},$ and
$\bar x \in M^f,$ with $\pi(\bar x)=x.$

We first deal with unstable sets. Given $x \in \overline{U},$ we
show that there exists $\bar x \in \pi^{-1}(x)$ such that
$W^u(\bar x) \subset \overline{U}.$

First consider $x \in U.$ Since $x$ is a non-wandering point,
there exist sequences $y_n \in U$ and $i_n \in \mathbb{N}$ such that
$
y_n \to x$ and
$f^{i_n}(y_n) \to x.$ Moreover, the sequence $(i_n)$ may be chosen increasing.

Consider $\eta >0$ such that $B_\eta(x) \subset U.$
After relabeling, we may assume that
\[
y_n \in B_\eta(x)
\qquad \text{and} \qquad
f^{i_n}(y_n) \in B_\eta(x),
\]
for every $n \geq 1.$
For each $n \geq 1,$ take $\bar{y}_n \in M^f$
such that $\pi(\bar{y}_n) = y_n.$

Since $\Lambda$ is hyperbolic, there exists $r>0$ such that
\[
W_r^u(\bar{y}_n) \subset B_\eta(x).
\]
By invariance of $U$ under $f,$ one obtains
\[
f^{i_n}(W_r^u(\bar{y}_n) \subset U.
\]

Choose $R>0$. Observe that, for all sufficiently large $n,$

\[
f^{i_n}(W_r^u(\bar{y}_n)) \supset
W_R^u(\bar f^{i_n}(\bar{y}_n)),
\]
by uniform expansion of $Df|_{E^u}$.

Since $M^f$ is compact, after passing to a subsequence,
we may assume that
$
\bar f^{i_n}(\bar{y}_n) \to \bar x,
$
where $\pi(\bar x)=x$.

Finally, by Theorem \ref{PRZ-1},

\[
U \supset W_R^u(\bar f^{i_n}(\bar{y}_n))
\xrightarrow{C^1}
W_R^u(\bar x).
\]

Hence
$
W_R^u(\bar x)\subset \overline U.
$
Since $R>0$ was arbitrary, we conclude that
$
W^u(\bar x)\subset \overline U.
$
Arguing as before, given $x \in \overline{U},$  one obtains
$
W^u(\bar x) \subset \overline{U},
$
for some $\bar x \in \pi^{-1}(x).$

\bigskip

We now turn to stable manifolds.
As $f$ is not invertible, arguments involving inverse branches require
more care than in the diffeomorphism setting.

Fix $x \in U$, and let
$B_\eta(x)$, $y_n$, and $i_n$ be as before.
Then there exists $r>0$ such that for every $n \geq 1,$

\[
W_r^s(y_n) \subset B_\eta(x)
\qquad \text{and} \qquad
W_r^s(f^{i_n}(y_n)) \subset B_\eta(x).
\]

Fix an arbitrary $R>0$. Taking inverse branches along the orbit of $y_n,$ the hyperbolicity of $f$
and the complete invariance of $U,$ we obtain

\[
U \supset
f^{-i_n}\Bigl(W_r^s(f^{i_{n}}(y_n))\Bigr)
\supset
W_R^s(y_n),
\qquad \forall n \geq 1.
\]

Moreover,

\[
W_R^s(y_n)
\xrightarrow{C^1}
W_R^s(x).
\]

Thus
$
W_R^s(x) \subset \overline{U}.
$

Similarly to the unstable case, given $x \in \overline{U},$
we conclude that

\[
W^s(x) \subset \overline{U},
\qquad \text{for every } x \in \overline{U}.
\]

To finish the proof, given $x \in \overline{U},$
consider $\bar x$ such that
$W^u(\bar x) \subset  \overline{U}.$
Define

\[
V(x)=\bigcup_{z \in W_r^u(\bar x)} W_r^s(z),
\]
where $r>0$ is sufficiently small.

The local product structure implies that $V(x)$ is open in $M$.
Moreover, $V(x)\subset \overline U$.
Therefore, $\overline U$ is both open and closed.
Since $M$ is connected,

\[
M = \overline{U} \subset \Lambda \subset \Omega \subset M,
\]

hence
$
\Lambda = M = \Omega.
$

Therefore, $f$ is an Anosov endomorphism with
$\Omega=M$, and hence it is transitive.
\end{proof}

\begin{proof}[Proof of Corollary \ref{cor1}]
Suppose that there exists $x \in \Omega(f)$ and an open set $V_x$ containing $x$ such that $V_x \subset \Omega(f)$. Then $\operatorname{Int}(\Omega(f)) \neq \varnothing$. By Theorem~\ref{TeoA}, it follows that $\Omega(f)=M$, and hence $\partial \Omega(f)=\varnothing$.

Otherwise, for every $x \in \Omega(f)$ and every open set $V_x$ containing $x$, we have $V_x \not\subset \Omega(f)$. This implies that every point $x \in \Omega(f)$ belongs to $\partial \Omega(f)$, and therefore $\partial \Omega(f)=\Omega(f)$.	
\end{proof}

We now present the proof of the Theorem \ref{TeoB}.

\begin{proof}[Proof of Theorem \ref{TeoB}]
Denote by $B$ the repeller of $f$ with nonempty interior, and consider $U=\operatorname{Int}(B)$. Since $B\subset \Omega(f)$, arguing as in the proof of Theorem~\ref{TeoA}, for every $x\in \overline{U}$ there exists an orbit $\bar{x}$ such that $W^u(\bar{x})$ is entirely contained in $\overline{U}$. Note that, for this the complete invariance is not needed.

We claim that $\overline{U}$ is open. Take $x\in \overline{U}$, since $W^s(B)=B$, there exists an open set
$V_r(x)\subset \displaystyle\bigcup_{z\in W^u(\bar{x})}W^s(z)$, contained in $B$, therefore $x\in \operatorname{Int}(B)=U$.

As $M$ is connected, it follows that $B=M$, and $f$ is a transitive Anosov endomorphism.
\end{proof}

\section{Transitivity and Accessibility}

In this section, we prove Theorem \ref{accBrin}, which is a variant of Brin's theorem (see Theorem 1.1 in \cite{Br}). The proof adapts Brin's argument to partially hyperbolic endomorphisms, leveraging accessibility and the fact that $\Omega(f) = M$.

The volume-preserving setting of this result was previously established in \cite{MU25}, where the volume-preserving hypothesis ensures that $\Omega(f) = M$. While their proof relies on the set of recurrent points having full volume.  A suitable modification allows us to obtain a topological version of theorem in \cite{MU25}. Although the underlying arguments are similar, we include the full proof here for completeness.

First, we recall the concept of recurrent point and present a classical result.

\begin{definition}
	Let $f: M \rightarrow M$ be a continuous map of a compact metric space $(M,d)$. A point $x \in M$ is called a \emph{recurrent point} of $f$ if for every neighborhood $V$ of $x$ there exists $n\geq 1$ such that $f^n(x)\in V.$
\end{definition}

Denote by $R(f) = \{x \in M\;| \; x  \in \omega_f(x)\},$ the set of recurrent points of $f.$

\begin{lemma}\label{Lema-reconr-residual}
Let $f:M\to M$ be a continuous map on a compact metric space $(M,d)$. If $\Omega(f)=M$, then the set of recurrent points $R(f)$ is residual in $M$.
\end{lemma}

\begin{proof}
 For each $k\geq 1$ consider $\varepsilon_k = 1/k$. Consider the set $A_k=\{x\in M \mid \text{there is } n\geq 1 \text{ with } d(f^n(x),x)<2\varepsilon_k\}.$

 Take an arbitrary $z\in M.$ Since $z\in\Omega(f),$ for a fixed $k\geq 1,$ there is $x\in B_{\varepsilon_k}(z)$ and $n\geq 1$ such that
	\begin{equation}\label{eq-dist}
		d(f^n(x),x)<2\varepsilon_k.
	\end{equation}
Since $f$ is continuous, fixed $n\geq 1$, the map $x \mapsto d(f^n(x),x)$ is also continuous, which implies that the set of points which satisfy property \eqref{eq-dist} is open. Now take $0 < r < \varepsilon_k$ and consider $B_r(z),$ as before there is a point $x \in B_r(z),$ for which there exists some $n \geq 1,$ such that $f^n(x) \in B_r(z),$ particularly $x \in A_k.$

The previous assertions ensure us that $A_k$ is open and dense set for each $k\geq 1.$  Finally, observe  $$R(f) = \bigcap_{k=1}^{\infty}A_k.$$ Being a countable intersection of open and dense sets, it follows from the Baire Category Theorem that $R(f)$ is residual in $M$ and in particular it is dense in $M.$
\end{proof}

The proof of Theorem \ref{accBrin} follows from the next two lemmas together with Lemma \ref{Lema-reconr-residual}.

\begin{lemma}\label{accs}
	Let $x = z_0, \ldots, z_n = y$ be as in Definition \ref{accdef} and let $z_{i+1} \in W^s(z_i)$. If every neighborhood $U$ containing $z_i$ intersects $f^k(B_{\varepsilon}(x))$ for some integer $k \geq 0$, then $f^{\hat{k}}(B_{\varepsilon}(x)) \cap B_{\varepsilon}(z_{i+1}) \neq \emptyset$ for some integer $\hat{k} \geq 0$.
\end{lemma}

The proof mirrors Lemma 2 of \cite{Br}, we write for the completeness of the text. The Figure \ref{picbrin} helps us to fix the scene.

\begin{proof}
	Let $R>d_{W^s}(z_i,z_{i+1})$ and $V$ be a neighborhood of $z_{i+1}$. By the continuity of $W^s$, there is a $\delta>0$ small enough such that for all $z\in B_{\delta}(z_i)$, we have $W^s_R(z)\cap V\neq\varnothing$.
	
	By hypothesis, there exists $k\in\mathbb{N}$ such that $U=f^k(B_{\varepsilon}(x))\cap B_{\delta}(z_i)$ is an open and nonempty set. The set $W^s_R(U)\cap V$ has a non-empty interior by the construction of $B_{\delta}(z_i)$. Denote $A=\operatorname{Int}(W^s_R(U)\cap V)$.
	
	Using Lemma \ref{Lema-reconr-residual}, consider $a\in R(f)\cap A$ and take $\alpha>0$ small enough such that $B_{\alpha}(a)\subset A$. Since $a$ is a recurrent point, there is a sequence $0<n_1<n_2<\cdots$ such that $f^{n_j}(a)\in B_{\alpha/2}(a)$. Since $a\in W^s_R(z)$ for some $z\in U$, there exists $j$ large enough such that $d(f^{n_j}(a),f^{n_j}(z))<\alpha/2$. This implies $d(f^{n_j}(z),a)  <  \alpha.$
	Therefore, $f^{n_j}(z)\in B_{\alpha}(a)\subset V$. Finally, since $z=f^k(b)$ for some $b\in B_{\varepsilon}(x)$, we have that $f^{k+n_j}(b)\in V$.
\end{proof}

	\begin{center}
\begin{figure}
	\includegraphics[scale=1]{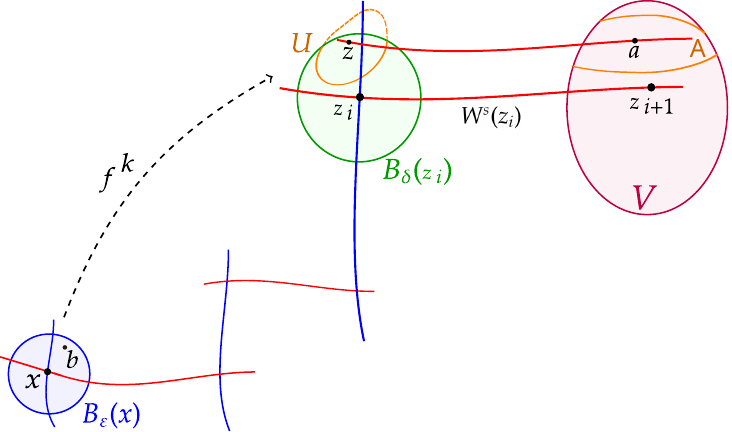}
\caption{Accessibility and transitivity}
	\label{picbrin}
\end{figure}
\end{center}

The key challenge lies in handling unstable leaves (see Lemma \ref{accu} below), since partially hyperbolic endomorphisms do not necessarily admit a well-defined unstable foliation on $M$. Unlike stable leaves, which are uniquely determined at each $x \in M$, unstable leaves depend on the choice of preorbit $\bar{x} \in M^f$ (with $p(\bar{x}) = x$). To overcome this issue, we work with projections of pieces of unstable leaves from the universal cover $\widetilde{M}$. In this cover, the lifted dynamics $\widetilde{f}$ acts as a diffeomorphism, ensuring locally consistent unstable manifolds that can be used to approximate the unstable leaves $W^u(\bar{x})$ on $M$.

\begin{lemma}\label{accu}
	Let $x = z_0, \ldots, z_n = y$ be as in Definition \ref{accdef}. Suppose that $z_{i+1} \in W^u(\bar{z}_i)$ such that $p(\bar{z}_i) = z_i$. If for any open set $U \subset M$ with $z_i \in U$, there is an integer $k = k(U) \geq 0$ such that $f^k(B_{\varepsilon}(x)) \cap U \neq \emptyset$, then there exists an integer $\hat{k} \geq 0$ such that $f^{\hat{k}}(B_{\varepsilon}(x)) \cap B_{\varepsilon}(z_{i+1}) \neq \emptyset$.
\end{lemma}

\begin{proof}
	First, we recall that for the partially hyperbolic system $f$, the unstable manifolds $W^u_f(\bar{x})$, where $\bar{x} = (\ldots, x_{-2}, x_{-1}, x_0, x_1, x_2, \ldots) \in M^f$, can be constructed using the graph method, as outlined in \cite{HPS}. Moreover, the lifts of $f$ are partially hyperbolic diffeomorphisms of $\widetilde{M}$ (see \cite{CM22}), for which both unstable and stable manifolds can be constructed. Additionally, the complete orbits of $\widetilde{f}$ project onto complete orbits of $M^f$.
	
	Fix a lift $\tilde{f}$. The unstable manifolds vary continuously in the $C^1$ topology by the graph method. The density in $M^f$ of orbits projected from orbits of $\widetilde{M}^{\tilde{f}}$ ensures that the unstable manifolds of $f$ can be $C^1$-approximated by the projections of the unstable manifolds of $\tilde{f}$ onto $M$ from the universal cover $\widetilde{M}$. Given a fixed $R > 0$, this approach can be carried out by considering unstable discs of constant radius.
	
	Project an unstable leaf $W^u(z_i) = \pi(W)$ from the universal cover $\widetilde{M}$, ensuring it is $C^1$-close to $W^u(\bar{z}_i)$ and intersects $B_{\varepsilon}(z_{i+1})$ at a point $w_{i+1}$ arbitrarily close to $z_{i+1}$.
	
	Let $U$ be a sufficiently small neighborhood of $z_i$ contained in $f^k(B_{\varepsilon}(x))$. Project a $u$-foliated set from $\widetilde{M}$ to $M$, such that unstable leaves starting in $U$ cross $B_{\varepsilon}(z_{i+1})$. By continuity, there exists $R > 0$ such that every $R$-radius unstable disc centered in $U$ intersects $B_{\varepsilon}(z_{i+1})$.
	
	For $\hat{z}_i = (\ldots, z_{i, -3}, z_{i, -2}, z_{i, -1}, z_{i, 0}, z_{i, 1}, z_{i, 2}, \ldots) \in M^f$ such that $p(W^u(\hat{z}_i)) = W^u(z_i)$, choose a large integer $N > 0$ and an open neighborhood $V_N \subset M$ of $z_{i, -N}$ such that $f^N(V_N) = U$. Define $V_N' = R(f) \cap V_N$. Since $f$ is a local diffeomorphism, $f^N(V_N')$ is residual in $U$. Moreover, taking $U' = R(f) \cap U$, we know that $U'$ is also residual in $U$. Consequently, the intersection $U'' = U' \cap f^N(V_N')$ remains residual in $U$. Observe that for a given point $\theta \in U''$, there is an increasing sequence $(n_j)_{j=1}^{+\infty}$ of positive integers such that $f^{n_j}(\theta) \in U$, and consequently $f^{n_j - N}(\theta) \in V_N$.
	
	To finish the proof, choose a point $\theta \in U''$ and let $W^u(\theta) \cap U$ be the connected component of the leaf of the previously constructed foliation that contains $\theta$. Consider the large unstable discs $f^{n_j - N}(W^u(\theta) \cap U)$. Since $N$ was taken enough large, the $R$-disc of $f^N(f^{n_j - N}(W^u(\theta) \cap U))$ centered at $f^{n_j}(\theta)$ is exactly the $R$-disc of $f^{n_j}(W^u(\theta))$ centered at $f^{n_j}(\theta)$. By the graph method, this is $C^1$-close to the $R$-disc $W^u(z_i)$ centered at $z_i$. Thus, for sufficiently large values of $j > 0$, we have $f^{n_j}(W^u(\theta) \cap U) \cap B_{\varepsilon}(z_{i+1}) \neq \emptyset$.
	
	Finally, this implies that $f^{k + n_j}(B_{\varepsilon}(x)) \cap B_{\varepsilon}(z_{i+1}) \neq \emptyset$ for large $j > 0$.
\end{proof}

\begin{proof}[Proof of Theorem \ref{accBrin}]
	
	To prove topological transitivity, we show that for any $x, y \in M$ and $\varepsilon > 0$, there exists an integer $k = k(x,y, \varepsilon) \geq 0$ such that
	$$
	f^k(B_{\varepsilon}(x)) \cap B_{\varepsilon}(y) \neq \emptyset.
	$$
	
	Since $f$ is accessible, there is a $C^0$-path connecting $x$ and $y$, composed of alternating stable and unstable segments. The theorem then follows by iteratively applying Lemmas \ref{accs} and \ref{accu} along this $C^0$-path from $x$ to $y$.
\end{proof}

\section{Proofs of theorems \ref{TeoC} and \ref{TeoD}}

\begin{proof}[Proof of Theorem \ref{TeoC}]
	Let $U=\operatorname{Int}(\Omega(f))$. Following the same arguments as in the proof of Theorem \ref{TeoA}, for all $x\in \overline{U}$, there is an orbit $\bar{x}$ such that its unstable leaf $W^u(\bar{x})$ is entirely contained in $\overline{U}$. Because $f$ is $u$-special, the unstable leaf $W^u(x)$ is unique and independent of the chosen prehistory of $x$. Similarly, we can deduce that the stable $W^s(x)$ is also contained in $ \overline{U}$.
	
	 We now claim that $\overline{U}=M.$ To see this, let $x\in M$ be an arbitrary point. Since $f$ is accessible, for any fixed $y\in \overline{U}$, there exists a finite path consisting of stable and unstable segments connecting $x$ to $y$. This means there is a sequence of points $y=y_0,y_1,y_2\ldots,y_n=x$ such that $y_{i+1}\in W^s(y_{i})$ or $y_{i+1}\in W^u(y_{i})$ for each $i=0,\ldots, n-1$. Since we have already established that the stable and unstable leaves of any point in $\overline{U}$ remain in $\overline{U}$, an inductive argument starting from $y_0$ shows that $y_i\in \overline{U}$ for $i=0,\ldots,n$. In particular, $x=y_n\in M$. Since $x$ was arbitrary, we conclude that  $\overline{U}=M$, which yields $\Omega(f)=M$. Finally, by Theorem \ref{accBrin}, it follows that $f$ is transitive.
\end{proof}

Finally, we will prove Theorem \ref{TeoD}.

\begin{proof}[Proof of Theorem \ref{TeoD}]
Denote $U=\operatorname{Int}(\Lambda)$. Because $\Lambda$ is completely $f$-invariant, Lemma \ref{Lem00} ensures that both $U$ and $\overline{U}$ are also completely $f$-invariant. Consequently, $\pi^{-1}(U)=U^f$ and $\bar{f}(U^f)=U^f=\bar{f}^{-1}(U^f)$, with analogous relations holding for $\overline{U}$ and $\Lambda$. Moreover, by Theorem \ref{Teo-transitive}, the transitivity of $f\vert_{\Lambda}$ implies that the induced map $\bar{f}\vert_{\Lambda^f}$ is also transitive.

Since $M^f$ is a compact separable metric space, there exists a point $\bar{y}\in U^f$ whose forward orbit $\mathcal{O}^+(\bar{y})=\{\bar{f}^{n}(\bar{y});n\in\N\}$ and backward orbit $\mathcal{O}^-(\bar{y})=\{\bar{f}^{-n}(\bar{y});n\in\N\}$ are both dense in $U^f$.

Now, consider an arbitrary point $x\in U$ and a prehistory $\bar{x}\in U^f$. Since $\mathcal{O}^+(\bar{y})$ is dense, we can find a sequence $(n_k)$ such that $\bar{f}^{n_k}(\bar{y})\to \bar{x}$. Choose $\eta>0$ small enough so that $B_{\eta}(x)\subset U$. For $l$ sufficiently large, the unstable disc $D^u_{\eta/2}:=D^u_{\eta/2}(\bar{f}^{n_l}(\bar{y}))$ is fully contained in $U$. Due to the complete invariance of $U$, the forward iterates of this disc remain in $U$, and their lengths grow indefinitely. Applying Theorem \ref{PRZ-1}, we see that these discs converge to $W^u(\bar{x})$. Given that $\overline{U}$ is invariant, it follows that $W^u(\bar{x})$ must be entirely contained in $\overline{U}$.

By a similar argument, the same holds for the stable leaf: by using the density of the backward orbit $\mathcal{O}^-(\bar{y})$ and taking a sequence $\bar{f}^{-m_k}(\bar{y})\to \bar{x}$, we conclude that $W^s(x)$ is also entirely contained in $\overline{U}$. Finally, because $f$ is accessible, we can argue exactly as in the proof of Theorem \ref{TeoC} to deduce that $M=\overline{U}\subset \Lambda$. This yields $M=\Lambda$, which concludes the proof that $f$ is transitive.

\end{proof}

%####################################################################
%####################################################################

\end{document}